\documentclass[11pt,leqno]{article}
\usepackage{graphicx, amsfonts, amsthm, amsxtra, amssymb, verbatim, makeidx}
\usepackage{subeqnarray, relsize}
\usepackage[mathscr]{euscript}
\usepackage{hyperref, tikz-cd}
\usepackage{aliascnt}
\usepackage[nameinlink,capitalize,noabbrev]{cleveref}
\hypersetup{
    colorlinks=true,       
    linkcolor=blue,          
    citecolor=blue,        
    filecolor=blue,      
    urlcolor=blue           
}

\newtheorem{prop}{Proposition}
\newtheorem{conj}{Conjecture}

\theoremstyle{remark}
\newtheorem{rem}{Remark}

\theoremstyle{definition}
\newtheorem{defi}{Definition}

\def\TS{{\mathlarger{\bf T}}}                

\def\NLL{{\rm NL}}   

\def\NLL{{\rm NL}}   

\def\Z{\mathbb{Z}}                   
\def\Q{\mathbb{Q}}                   
\def\C{\mathbb{C}}                   
\def\N{\mathbb{N}}                   
\def\dR{{\rm dR}}                    
\def\GL{{\rm GL}}                

\def\X{{\sf X}}                      
\def\T{{\sf T}}                      

\def\Ts{{\sf S}}

\def\codim{{\rm codim}}                

\def\prim{{\rm  0}}                  

\def\P{\mathbb{P}}

\def\O{{\cal O}}                     

\def\Resi{{\rm Resi}}              

\def\cf{r}   

\makeindex
\makeglossary
\begin{document}

\begin{center}
{\LARGE\bf
On a counterexample to a conjecture of J. Harris for octic surfaces
}
\\
\vspace{.1in} {\large {\sc Hossein Movasati}}
\footnote{
Instituto de Matem\'atica Pura e Aplicada, IMPA, Estrada Dona Castorina, 110, 22460-320, Rio de Janeiro, RJ, Brazil,
{\tt \href{http://w3.impa.br/~hossein/}{www.impa.br/$\sim$hossein}, hossein@impa.br.}}
\end{center}


\begin{abstract}


We take a sum $C_1+\cf C_2,\ \cf\in\Q$ of a line $C_1$ and a complete intersection curve $C_2$ of type
$(3,3)$ inside the octic Fermat surface and with no intersection points.  We gather strong evidences to the fact that for all except a finite number of $\cf$,
the Noether-Lefschetz loci attached to the cohomology classes of $C_1+\cf C_2$ are set theoretically distinct $31$ codimensional subvarieties intersecting each other  in a $32$ codimensional subvariety of the ambient space.
The maximum  codimension for components of the Noether-Lefschetz locus in this case is $35$, and hence, we provide a possible counterexample to a conjecture of J. Harris.
\end{abstract}

\section{Introduction}

In the parameter space $\T$ of smooth surfaces of degree $d\geq 4$ in $\P^3$ the Noether-Lefschetz locus $\NLL_d$ is a union of enumerable subvarieties of $\T$ and its points parameterize surfaces with Picard number $\geq 2$. A component of $\NLL_d$  of codimension equal to (resp. strictly less than) $h^{20}=\binom{d-1}{3}$ is called general (resp. special). It is known that general components are
dense in $\T$ in  both usual and Zariski topology, see \cite{CHM88}, \cite[\S 5.3.4]{vo03} and \cite{CL1991}, and special  components of codimension $d-3$  and $2d-7, \ d\geq 5$ are unique and parameterize respectively surfaces with a line and conic, see \cite{green1988, green1989, voisin1988, voisin89}. This implies that for $d=5$ we have only two special components. J. Harris in 1980's conjectured that the number of special components must be
finite. C. Voisin in \cite{voisin1991} found counterexamples to this for a large $d$, however, the conjecture in lower degrees remains open. In \cite{voisin90} it is  proved that for $d=6,7$ the number of reduced special components
is  finite, and so, it is expected that Harris' conjecture is true in these cases.  In \cite{ho2019} we have described a conjectural description of an infinite number of components of $\NLL_8$ of codimension $31$, whereas the general components have codimension $35$. In the present article we gather more evidences to this conjecture.

Let $C_1$ and $C_2$ be respectively a line (complete intersection of type $(1,1)$) and a complete intersection of type $(3,3)$ inside the Fermat octic $X_0: x_0^8+x_1^8+x_2^8+x_3^8=0$ obtained by canonical factorization of $x_0^8+x_1^8$ and $x_2^8+x_3^8$. Further assume that $C_1$ does not intersect $C_2$. For a rational number $r\in\Q,\ r\not=0$ we consider
the cohomology class of 
$$
\delta_0:=[C_1]+r[C_2]\in H^2(X_0,\Q)
$$ 
and its parallel transport $\delta_t\in H^2(X_t,\Q),\ t\in (\T,0)$, that is, $\delta_t$ is a flat section of the Gauss-Manin connection. Here, $\T$ is the full parameters space of smooth octic surfaces. The local Noether-Lefschetz locus $V_{[C_1]+r[C_2] }\subset \T$ parameterizes those $X_t$ such that $\delta_t$ is still of type $(1,1)$.
It has a natural analytic space structure (ring of functions might have zero divisors and nilpotent elements).
Let $V_{C_i}$ be the local branch of the subspace of $\T$ parameterizing a surface together with a line for $i=1$ and a complete intersection of type $(3,3)$ for $i=2$, corresponding to deformations of $(X_0, C_i)$. In the present article we gather strong evidences to the following conjecture.
\begin{conj}
\label{20052026rio}
The local Noether-Lefschetz loci $V_{[C_1]+r[C_2] },\ r\in\Q,\ r\not=0$ are smooth, distinct analytic spaces of codimension $31$  passing through  the common $32$-codimensional subvariety $V_{C_1}\cap V_{C_2}$. In particular,  the underlying analytic varieties are distinct and of codimension $<35$.
\end{conj}
The parameter space $\T$ is of dimension $165$ and it is hard to perform computations on this.
By $\GL(4,\C)$ action on $\T$, we can reduce the above conjecture to $149$ dimensional subspace of $\T$ which is still big. 
 For a natural number $N\geq 2$, we say that $V_{\delta_0}$ is $N$-smooth if the infinitesimal local Noether-Lefschetz locus $V^N_{\delta_0}$ is the $N$-jet of a smooth  variety at $0$, for further details see \cite{ho2019}. In this generality, we can prove that  $V_{\delta_0}$ is $2$-smooth, see \cref{08062026tulivizinho}. For a subscheme $\Ts$ of $\T$ containing the point $0$, we define $V_{\delta_0}^\Ts$ and  $V_{\delta_0}^{\Ts,N}$ in a similar way and so $V_{\delta_0}= V_{\delta_0}^\T$ and   
 $V_{\delta_0}^N= V_{\delta_0}^{\T,N}$. 
In \cite[Theorem 2]{ho2019} we have taken a $32$-dimensional subspace $\Ts$ of $\T$ transversal to  $V_{C_1}\cap V_{c_2}$ at the Fermat point $0$ and have  proved that $V_{[C_1]+r[C_2]}^{\Ts}$ is $5$-smooth. In this article we give a full theoretical proof of the smoothness of  $V_{[C_1]+r[C_2]}^\Ts$, see \cref{16062026impasara}. As this is not enough to prove \cref{20052026rio}, we observe  that \cref{20052026rio} implies the smoothness of $V_{\delta_0}^{\check\Ts}$, where $\check\Ts$ is the $11$ dimensional subscheme of $\T$ given by
\begin{equation}
\label{danithegamer2019}
X_t: \ \ x_0^8+x_1^8+x_2^8+x_3^8-\sum_\alpha t_\alpha x^\alpha=0, \ \ t\in\check \Ts,
\end{equation}
where $x^\alpha$ runs through the monomials 
\begin{equation}
\label{04062026kevin}
x_1^4x_3^4,
   x_0x_1^3x_3^4,
   x_0^2x_1^2x_3^4,
   x_1^4x_2x_3^3,
   x_0x_1^3x_2x_3^3,
   x_0^2x_1^2x_2x_3^3,
   x_1^4x_2^2x_3^2,
   x_0x_1^3x_2^2x_3^2,
   x_0^2x_1^2x_2^2x_3^2,
   x_0^3x_1x_3^4, 
\end{equation}
and a fixed monomial $x_0^{\alpha_0}x_1^{\alpha_1},\ \alpha_0+\alpha_1=8$ of degree $8$. As this parameter space has $11$ dimensions, we can do higher order approximation of periods, and prove that
$V_{\delta_0}^{\check\Ts}$ is $9$-smooth, see \cref{16062026viana}. 

\begin{rem}\rm
 It might happen that all the local Noether-Lefschetz loci,  $V_{[C_1]+\cf [C_2]}$ are inside an irreducible component of $\NLL_8$ of codimension $<31$. Whether this happens or not  seems to be a hard question.
\end{rem}

\section{A special family}
Let $d\geq 2$ be a natural  number and let us consider the following deformation of the Fermat surface of degree $d$
\begin{equation}
\label{20260520propina1}
X_t: \ \ \P\{f_t=0\},\ \ \    f_t:=x_0^d+x_1^d+x_2^d+x_3^d-\sum t_\alpha x^\alpha,\ \ t=(t_\alpha)\in\T,\ \
\end{equation}
where $x^\alpha$ runs through all degree $d$ monomials. We have the following line inside the Fermat surface:
\begin{equation}
\label{20260520propina2}
 {\mathbb P}^1:
\left\{
 \begin{array}{l}
 x_{0}-\zeta_1x_{1}=0,\\
 x_{2}-\zeta_2 x_{3}=0,
 \end{array}
 \right.\ \ \ \ \ \zeta_1^d=\zeta_2^d=-1,
 \end{equation}
with rational coefficients. We will need the following integral formula frequently:
\begin{equation}
\label{churrasco2019}
\mathlarger{\mathlarger{\int}}_{\P^1}
\Resi\left(
\frac{x_0^{\beta_0}x_1^{\beta_1}x_2^{\beta_2} x_{3}^{\beta_3}\cdot  \sum_{i=0}^{3}(-1)^ix_i\widehat{dx_i}}
     {( x_0^{d}+x_1^{d}+x_2^{d}+x_{3}^d)^{2}}\right)
=\left\{
\begin{array}{cc}
\frac{2\pi \sqrt{-1}}{d^2} \zeta_1^{\beta_0+1}\zeta_2^{\beta_2+1}  & \hbox{ if } \beta_0+\beta_1=\beta_2+\beta_3=d-2\\
0 & \hbox{otherwise}
\end{array}\right.
\end{equation}
see \cite[Theorem 1]{roberto}.
We have the residue map $\Resi : H^3_\dR(\P^3\backslash X_t)\to H^2_\dR(X_t)_\prim$ and  $H^2_\dR(X_t)_\prim$  ($\prim$ stand for primitive) is generated by differential forms
$$
\omega_\beta:=\Resi\left(\frac{
 x_0^{\beta_0} x_1^{\beta_1}x_2^{\beta_2}x_{3}^{\beta_{3}}
 \left(  \sum_{i=0}^3 (-1)^ix_i\widehat{dx_i}  \right)
 }{f_t^k}\right),\ \ \ \beta_0+\beta_1+\beta_2+\beta_3=kd-4. 
$$
From now on we use $\delta_0\in H_2(X_0,\Z)$ and  discard the usage of $\delta_0$ in cohomology.
A local Noether-Lefschetz locus $V_{\delta_0}$ for a Hodge cycle $\delta_0\in H_2(X_0,\Q)$ is give by the zero set of the
ideal
$$
I_{\delta_0}:=\left \langle \int_{\delta_t}\omega_\beta,\ \beta_0+\beta_1+\beta_2+\beta_3=d-4\right\rangle\subset \O_{\T,0}.
$$
where $\O_{\T,0}$ is the $\C$-algebra of holomorphic functions in a neighborhood of $0$. 
Now,  we consider the case $d$ even  and consider a subscheme $\check\T$ of $\T$ such that
$x^\alpha$ runs through monomials of the form
\begin{equation}
\label{12062026saraagua}
x_0^{\alpha_0}x_1^{\alpha_1}x_2^{\alpha_2}x_3^{\alpha_3},\ \ \alpha_0+\alpha_1=\alpha_2+\alpha_3=\frac{d}{2}.
\end{equation}
Therefore, $\check \T$ is of dimension $(\frac{d}{2}+1)^2$.
Let also $C$ be a divisor in $X_0$ which is a combination of lines of the form \eqref{20260520propina2} with varying $\zeta_1$ and $\zeta_2$.
Let $\delta_{t}\in H_2(X_t,\Z),\ t\in(\T,0)$ be the monodromy (parallel transport) of the cycle
$\delta_0:=[C]\in H_2(X_0,\Z)$ along a path which
connects $0$ to $t$.

\begin{prop}
\label{21052026estacionamentodemerda}
For a monomial
$$
x^\beta=x_0^{\beta_0} x_1^{\beta_1}x_2^{\beta_2}x_{3}^{\beta_{3}}, \ \beta_0+\beta_1+\beta_2+\beta_3=kd-4,\ \ \beta_0+\beta_1\not=k\frac{d}{2}-2
$$
we have
 \begin{equation}
 \mathlarger{\mathlarger{\int}}_{\delta_t}\omega_\beta =0,\ \ t\in\check \T.
\end{equation}
\end{prop}
\begin{proof}
 Let $t_\alpha$ be the coefficient of the monomial $x^\alpha$ in the deformation \eqref{20260520propina1}. We have
 $$
 \frac{\partial}{\partial t_\alpha}\int_{\delta_t}\omega_{\beta}=k\int_{\delta_t}\omega_{\beta+\alpha}.
 $$
 Therefore, we only need to prove the proposition for $t=0$, that is, the Fermat surface. The affirmation now follows from the integral formula \eqref{churrasco2019}. 
\end{proof}
\begin{rem}\rm
 The whole discussion in this section is valid for a smooth  surface of the form $P(x_0,x_1)+Q(x_2,x_3)=0$ instead of the Fermat surface, where $P,Q$ are two homogeneous polynomial of degree $d$.
\end{rem}
\begin{rem}\rm
 \cref{21052026estacionamentodemerda} suggests that the subspace $H_\C$ of $H^2_\dR(X_t)_\prim,\ t\in\check\T$ given by the differential forms
 $$
 \omega_{\beta},\ \beta_0+\beta_1=\beta_2+\beta_3=k\frac{d}{2}-2
 $$
form a  variation of Hodge structures on $\check\T$ with Hodge numbers $(\hat h^{20},\hat h^{11},\hat h^{02})=((\frac{d}{2}-1)^2,(d-1)^2,  (\frac{d}{2}-1)^2)$.
For this we must prove that there is a $(d-1)^2+2(\frac{d}{2}-1)^2$ dimensional subspace $H_{\Q}$ of $H^2(X_t,\Q)$ such that $H_\C=H_{\Q}\otimes_{\Q} H_\C$.
By \cref{21052026estacionamentodemerda} we only know that the cohomology classes of the lines \eqref{20260520propina2} with varying $\zeta_1$ and $\zeta_2$ are in $H_\Q$.
For our main purpose, we do not need to prove this. We only need the fact that a local Noether-Lefschetz locus
$V_{\delta_0}^{\check\T},\ \delta_0:=[C]$ is given by $\hat h^{20}:=(\frac{d}{2}-1)^2$ equations, whereas $V_\delta^\T$ is given by $h^{20}=\binom{d-1}{3}$ equations.
\end{rem}
\begin{defi}\rm
We say that $V_{\delta_0}^{{\check\T}}$ (resp. $V_{\delta_0}^{\T}$) is very
general if $\codim(\TS_0 V_{\delta_0}^{{\check\T}})=\hat h^{20}$ (resp.  $\codim(\TS_0 V_{\delta_0}^{\T})= h^{20}$ ).
\end{defi}
In this article we treat a cycle $\delta_0$ such that $V_{\delta_0}^{\T}$ is not very general, however, $V_{\delta_0}^{{\check\T}}$ is general. For this reason it is useful to know the following proposition. 
\begin{prop}
\label{08062026ze1000}
Let $\delta_1, \delta_2\in H_2(X_0,\Q)$ be two cycles with non-zero primitive part.
If the locus $V_{\delta_1+\cf\delta_2}^{{\check\T}}$ is very general for at least one $\cf\in\Q$ and $\TS_0V_{\delta_1}^{\check\T}\not\subset \TS_0V_{\delta_2}^{\check\T}$
then the underlying analytic varieties of
$V_{[C_1]+\cf[C_2]}^\T,\ \cf\in\Q$ have infinite number of irreducible components $V_i,i\in\N$ with no inclusion between them.

\end{prop}
\begin{proof}
For a pencil of local Noether-Lefschetz loci $V_{\delta_1+\cf \delta_2}^{{\check\T}}$, since
the function $\cf\mapsto \codim (\TS_0V_{\delta_1+\cf\delta_2}^{\check\T} )$ is lower semi-continuous, we know that if
for one $\cf$ it reaches its maximum $\hat h^{20}$, then
for all except a finite number of $\cf$ we have $\codim (\TS_0V_{\delta_1+\cf\delta_2}^{\check\T})=\hat h^{20}$, and hence,
$V_{\delta_1+\cf\delta_2}^{\check\T}$ is smooth, and in particular it is  reduced.
Moreover, for two rational such numbers $\cf_1,\cf_2$ with $\cf_1\not= \cf_2$ we have
$$
\TS_0V_{\delta_1+\cf_1\delta_2}^{\check\T}\cap \TS_0V_{\delta_1+\cf_2\delta_2}^{\check\T}=\TS_0V_{\delta_1}^{\check\T}\cap \TS_0V_{\delta_2}^{\check\T}
$$
and the codimension of this vector space is bigger than $\hat h^{20}$. This follows from our hypothesis
$\TS_0V_{\delta_1}^{\check\T}\not\subset \TS_0V_{\delta_2}^{\check\T}$ and $\codim \TS_0V_{\delta_1}^{\check\T}= \hat h^{20}$.
Therefore, there cannot be inclusion between $\TS_0V_{\delta_1+\cf\delta_2}$ for all $\cf$ except a finite number.
This implies that the vector spaces $\TS_0V_{\delta_1+\cf\delta_2}^{\check\T}$ form a pencil  with the axis $\TS_0V_{\delta_1}\cap \TS_0V_{\delta_2}$. We have proved that intersecting the underlying analytic varieties of  $V_{\delta_1+\cf_1\delta_2}^\T$ with ${\check\T}$, we get infinite number of analytic
varieties with no inclusion between them. This implies the desired statement.
\end{proof}

\section{Evidences to  \cref{20052026rio}}
Let us write $x_0^8+x_1^8=f_1f_3f_{4},\ \deg(f_i)=i$,  $x_2^8+x_3^8=g_1g_3g_{4},\ \deg(g_i)=i$ and take the algebraic cycles
$C_1: f_1=g_1=0$ and $C_2: f_3=g_3=0$ inside the Fermat surface $X_0$.
The first tiny evidence to \cref{20052026rio} is the following:
\begin{prop}
\label{08062026tulivizinho}
  For the full family $\X\to\T$ of octic surfaces,   $V^\T_{[C_1]+\cf[C_2]}$ is $2$-smooth.
\end{prop}
The proof is purely computational, see \cref{07062026joinville}. Since $V^{\check \T}_{[C_1]+\cf[C_2]}=\check\T\cap V_{[C_1]+\cf [C_2]}^\T$, the next evidence is:
\begin{prop}
\label{16062026impasara}
The local Noether-Lefschetz locus  $V_{[C_1]+\cf[C_2]}^{\check\T},\ \cf\in\Q,\ \cf\not=0$ is a general pencil.
\end{prop}
This follows from the tangent space computations:
\begin{equation}
\label{26052026telhas2}
\codim \TS_0V_{C_1}^{\check\T}=1, \ \codim \TS_0V_{C_2}^{\check\T}=9,\ \
\codim \TS_0V_{[C_1]+\cf[C_2]}^{\check\T}=9, \ \ \codim \TS_0V_{[C_1]}^{\check\T}\cap \TS_0V_{[C_2]}^{\check \T}=10,
\end{equation}
It is useful to know that
 \begin{equation}
 \label{26052026telhas1}
 \codim \TS_0V_{C_1}^{\T}=5, \ \codim \TS_0V_{C_2}^{\T}=27,\ \
\codim \TS_0V_{[C_1]+\cf [C_2]}^\T=31,\ \  \codim \TS_0V_{[C_1]}\cap \TS_0V_{[C_2]}^\T=32
\end{equation}
For all these one can use the integral formula \eqref{churrasco2019} and the the computation of the tangent space of a Hodge locus in terms of IVHS, see \cite[5.3.3]{vo03} or \cite[Theorem 16.2]{ho13}, see the computer code in \cref{07062026joinville}.

Combining \cref{16062026impasara} and \cref{08062026ze1000}, we conclude that the local Noether-Lefschetz loci $V_{[C_1]+\cf [C_2]}^\T,\ \cf\in\Q$ have infinite number of irreducible components of codimension $\geq 31$. They might be of codimension $35$ which is not interesting if we are looking for a counterexample to Harris' conjecture.

Let $V_{C_1}, V_{C_2}\subset (\T,0)$ be local analytic varieties parameterizing deformations of $X_0$ together with line $C_1$, respectively with the curve $C_2$.
It is also an easy exercise in algebraic geometry to see that $V_{C_1}$ and $V_{C_2}$ are smooth at $0$, intersect each other transversely at $0$ and
$$
\codim(V_{C_1})=\codim(\TS_0 V_{C_1})=5,\  \codim(V_{C_2})=\codim(\TS_0 V_{C_2})=27,  \
$$
$$
\codim(V_{C_1}\cap V_{C_2})= \codim(\TS_0 (V_{C_1}\cap V_{C_2}))    =32.
$$
Note that $V_{C_1}\cap V_{C_2}$ is a local branch of the space of octic surfaces containing a line (complete intersection of type $(1,1)$) and a complete intersection of type $(3,3)$ with no intersection point.

From now on assume that $V^{\T,N-1}_{[C_1]+\cf [C_2]}$ is smooth but not $V^{\T,N}_{[C_1]+\cf [C_2]}$. By \cref{08062026tulivizinho} we have $N\geq 3$.
Let us take $35$ equations $f_1, f_2,\cdots, f_{35}\in \O_{\T,0}$ of the local Noether-Lefschetz locus $V_{[C_1]+\cf[C_2]}$ for the full parameter space of octic surfaces. We separate $31$ equations $f_1, f_2,\cdots, f_{31}$ with linearly independent linear parts.
By definition of $(N-1)$-smoothness, see \cite[page 298,  18.28]{ho13}, we can replace the rest of the equations $f_i,\ i=32,33,34,35$ by
new ones $\check f_1,\check f_2,\check f_3,\check f_4$ with a zero of order at least $N$ at $0$. This means if we write the Taylor series of $\check f_i$ at $0$, then it starts with a homogeneous polynomial of degree $N$ in $t$.

 The  smooth analytic variety $V_{C_1}\cap V_{C_2}$ is inside the smooth variety $V: f_1=\cdots=f_{31}=0$ is of codimension $1$, therefore, there is $f\in \O_{\T,0}$ with non-zero linear part such that $V_1:=V_{C_1}\cap V_{C_2}$ is given by
$f_1=\cdots=f_{31}=f=0$.
We know that $V_1$ is inside $V_{[C_1]+\cf [C_2]}$. This implies that $f|_V$ divides $\check f_i|_V$'s and so
$$
\check f_{i+31}=fh_i,\ i=1, 2,3,4,\hbox{module the ideal } \langle f_1,\cdots,f_{31}\rangle\subset \O_{\T,0},\ \hbox{ for some }  \ h_i\in \O_{\T,0}.
$$
All $h_i$'s cannot be in the ideal $\langle f_1,f_2,\ldots,f_{31}\rangle$, otherwise, $V_{[C_1]+\cf [C_2]}^{\T, N}$ is smooth. Moreover, the Taylor series of $h_i$'s starts at degree $N-1$. We conclude that
$$
\langle f_1,f_2,\ldots, f_{35}\rangle=\langle f_1,\ldots,f_{31},fh_1,fh_2,fh_3, fh_4\rangle,
$$
which implies that the underlying analytic variety of $V^\T_{[C_1]+\cf [C_2]}$ is a union of two varieties $V_1$ and
$V_2={\rm Zero}\langle f_1,\ldots,f_{31}, h_1,\cdots,h_4\rangle$.
In summary, the hypothesis that 
$V^{\T,N-1}_{[C_1]+\cf [C_2]}$ is smooth but not $V^{\T,N}_{[C_1]+\cf [C_2]}$ provides us with a local Noether-Lefschetz locus whose underlying analytic variety is inside a smooth analytic variety $V$ of codimension $31$,  and it is the union of $V_1$ and $V_2$ which are proper analytic subspaces of $V$,  $V_1$ is smooth and of codimension one in $V$ and $V_2$ is singular at $0$ and might not be irreducible or reduced.  By definition of $\Ts$ in \cite[Theorem 2]{ho2019}, the intersection of $V_{C_1}\cap V_{C_2}$ with $\Ts$ is a point, moreover, the restriction of $V^\T_{[C_1]+\cf [C_2]}$ to $\Ts$ is smooth. All these implies that restricted to  $\Ts$ we have $V=V_2$. 

We would like to get more evidence to the fact that $V^\T_{[C_1]+\cf [C_2]}$ is smooth, and in particular it is irreducible. For this, we take a slice $\check \Ts$ of $\T$  with non-zero dimensional intersection with both $V^\Ts_{[C_1]+\cf [C_2]}$ and $V_{C_1}\cap V_{C_2}$ and such that $\TS_0\check \Ts$ intersects $\TS_0V^\T_{[C_1]+\cf [C_2]}$ transversely. The ($N-$) smoothness of $V^\T_{[C_1]+\cf [C_2]}$ implies  the ($N$-) smoothness of  $V^{\check \Ts}_{[C_1]+\cf [C_2]}$. In order to achieve the latter  for large $N$, the dimension of such a slice must be low.  We take the family \eqref{danithegamer2019},
where the sum runs through the collection of $10$ monomials in \eqref{04062026kevin} 
and a set $M$ of monomials of the form $x_0^ax_1^{8-a},\ \ a=0,1,\ldots,8$.
The set of monomials \eqref{04062026kevin}, is the intersection of the set of monomials \eqref{12062026saraagua} and \cite[before Theorem 2]{ho2019}. For the corresponding parameter space, let us call it $\check \Ts$, we have
$$
\codim(\TS_0V^{\check\Ts_1}_{[C_1]+\cf[C_2]})=9
$$
and hence $V^{\check\Ts_1}_{[C_1]+\cf[C_2]}$ is a general pencil.
\begin{prop}
\label{16062026viana}
 For the family above with the $M$ all the monomials
 \begin{equation}
\label{12062026ezterab}
 x_0^{a}x_1^{8-a},\ \ a=0,1,\ldots,8.
\end{equation}
 $V^{\check \Ts}_{[C_1]+\cf [C_2]}$ is $6$-smooth. For $M$ a single monomial  $x_0^{a}x_1^{8-a}$,  $V^{\check \Ts}_{[C_1]+\cf [C_2]}$ is $9$-smooth.
\end{prop}
\begin{proof}
 See the computer code in \cref{07062026joinville}. 
\end{proof}

\section{Usage of LLM's}
Recently LLM's have become powerful tools which have learned almost all algorithms and reasoning available in the web. They also learn very fast content of new mathematics, as far as such a new mathematics relies on classical concept such as Taylor series. For this reason May 16, 2026,  I asked the following question from an LLM.
What  is the largest $N$ such that you can  prove \cite[Theorem 2]{ho2019} for $N$-th order approximation of the Noether-Lefschetz locus? This theorem is proved for $N=5$.  As the amount of computations for $N\geq 6$ increases exponentially, the LLM learned the content of \cite{ho2019} well, and in particular, it found out that the $32$-dimensional parameter space \cite[ (9) ]{ho2019} contains the  $10$-dimensional parameter space in \eqref{04062026kevin}.
This parameter space  is the intersection of the set of monomials in \cite[(9)]{ho2019} and those in \eqref{20260520propina1}. For this parameter space, let us call it $\check\Ts$, we can verify easily that $\codim(\TS_0V^{\check\Ts}_{[C_1]+\cf[C_2]})=9$, and hence, $V^{\check\Ts}_{[C_1]+\cf[C_2]}$ is a general pencil of curves. Therefore, the answer
the above question is $N=\infty$.
This was the beginning of my excitement with LLM's, however, the rest of my questions from the LLM only produced hallucination both for LLM and myself. Therefore, the author stopped asking questions from LLM's and started to write this article and the computer code in the next section without any external assistance.


\section{The computer code}
\label{07062026joinville}
Here is the computer code used for computations in this article. It is
written in {\sc Singular},  see \cite{GPS01}. It relies on the  the library
{\tt foliation.lib} which can be found in \href{http://w3.impa.br/~hossein/mod-fol-exa.html}
{the author's} or \href{https://github.com/movasati/NoetherLefschetz/blob/master/foliation.lib/}{\tt github}'s webpage.
{\tiny
\begin{verbatim}
LIB "foliation.lib";
int n=2; int d=8;  int tru=8;   //--truncation of periods at degree tru
int gene=11; //--a generic coefficient;
int dhalf= d div 2;
int i; int j; list ll;list Periods; list lcycles; intvec besh; list pkom; list lMpij;
intvec mlist=d;  for (i=1;i<=n; i=i+1){mlist=mlist,d;}

ring r=(0,z), (x(1..n+1)),dp;
poly cp=cyclotomic(2*d); int degext=deg(cp) div deg(var(1));
cp=subst(cp, x(1),z); minpoly =number(cp);
//-------------------Cycle (3,3)+rationalnumber*Cycle (3,3)---------------
int d1=3; int d2=3; int s1=1; int s2=1; int m1=0; int m2=0;
//---periods of 10 linear cycles
ll=TwoCI(n,d,intvec(d1,d2),intvec(s1,s2),intvec(m1,m2));
   lcycles=ll[1]; besh=ll[2]; Periods=list();
   for (i=1; i<=size(lcycles); i=i+1)
        {
        Periods=insert(Periods,
        PeriodLinearCycle(mlist, lcycles[i][1], lcycles[i][2],par(1)), size(Periods));
        }
//------------------periods of (3,3) and (1,1)----
pkom=list();
    for (i=1; i<=size(besh)-1; i=i+1)
         {
                            for (j=besh[i]+1; j<=besh[i+1]-1; j=j+1)
                                {
                                Periods[besh[i]]=Periods[besh[i]]+Periods[j];
                                }
                            pkom=insert(pkom, Periods[besh[i]], size(pkom));
        }
Periods=pkom;
//------------------IVHS and p_{i+j} matrix for full parameter space---------
lMpij=list();
for (i=1; i<=size(Periods); i=i+1)
    {
    lMpij=insert(lMpij, Matrixpij(mlist, Periods[i]), size(lMpij));
    }
rank(lMpij[1]), rank(lMpij[2]), rank(lMpij[1]+gene*lMpij[2]),
                             rank(concat(transpose(lMpij[1]), transpose(lMpij[2])));
//-------------------IVHS and p_{i+j} matrix for check T parameter space---------
list lmonx;   list R1=Monomials(list(var(1)), d div 2 ,1); list R2= Monomials(list(var(2),var(3)), d div 2,2)[(d div 2)+1]; R1; R2;
for (i=1; i<=size(R1); i=i+1)
    {for (j=1; j<=size(R2); j=j+1){ lmonx=insert(lmonx,R1[i]*R2[j], size(lmonx));}}
lMpij=list();
for (i=1; i<=size(Periods); i=i+1)
    {
    lMpij=insert(lMpij, Matrixpij(mlist, Periods[i],lmonx), size(lMpij));
    }
rank(lMpij[1]), rank(lMpij[2]), rank(lMpij[1]+gene*lMpij[2]),
                             rank(concat(transpose(lMpij[1]), transpose(lMpij[2])));
//------------------The rest of the code is taken from the main body of SmoothReduced procedure from foliations.lib---
//-------------------int zb=10;    ll=SmoothReduced(mlist,tru, lcycles, intvec(1,-zb), intvec(zb,zb), besh,0);
//-----------------------------------------------------
list wlist;         //--weight of the variables
for (i=1; i<=size(mlist); i=i+1)
    { wlist=insert(wlist, (d div mlist[i]), size(wlist));}

ring r2=(0,z), (x(1..n+1)),wp(wlist[1..n+1]);
poly cp=cyclotomic(2*d); degext=deg(cp) div deg(var(1));
cp=subst(cp, x(1),z); minpoly =number(cp);
//----No need to transfer the data from ring r to r2. All the data are integer valued---only Periods---
list Periods=imap(r, Periods);
list llm=MixedHodgeFermat(mlist);
list BasisDR;            for (i=1; i<=size(llm[1]); i=i+1)  { BasisDR=BasisDR+llm[1][i];}
list Fn2p1;              for (i=1; i<=n div 2; i=i+1)  { Fn2p1=Fn2p1+llm[1][i];}
//-------------------------------------------------------------------------------------
//---------------------Here you can give any list of monomials for the deformation space.
//----Full family    list lmonx=llm[4];
//---Two parameter family  i=size(llm[4]); list lmonx=llm[4][random(1,i)], llm[4][random(1,i)];

//---32 dimensional space perpendicular to VC1 \cap VC2
//----list lmonx=x(1)^6*x(3)^2, x(1)^6*x(2)*x(3), x(1)^5*x(3)^3, x(1)^4*x(3)^3, x(1)^6*x(2)^2, x(1)^5*x(2)*x(3)^2, x(1)^4*x(2)*x(3)^2, x(1)^4*x(3)^4,
//----x(1)^3*x(3)^4, x(1)^2*x(3)^4,    x(1)^5*x(2)^2*x(3),    x(1)^4*x(2)^2*x(3),    x(1)^4*x(2)*x(3)^3,    x(1)^3*x(2)*x(3)^3,    x(1)^2*x(2)*x(3)^3, x(3)^5,
//----x(1)^2*x(2)*x(3)^4, x(1)^3*x(2)*x(3)^4, x(1)^4*x(2)^2*x(3)^2, x(1)^3*x(2)^2*x(3)^2, x(1)^2*x(2)^2*x(3)^2, x(1)*x(2)*x(3)^4, x(1)^2*x(2)^2*x(3)^3,
//----x(1)^3*x(2)^2*x(3)^3, x(2)^2*x(3)^4, x(1)^2*x(2)*x(3)^5, x(1)*x(2)^3*x(3)^3, x(1)^5*x(3)^2, x(1)^3*x(3)^3, x(1)*x(3)^4, x(1)*x(2)^2*x(3)^3, x(2)^3*x(3)^3;


//---10 monomials+8 monomials-----
//---list lmonx=x(1)^4*x(3)^4, 1*x(1)^3*x(3)^4, 1^2*x(1)^2*x(3)^4, x(1)^4*x(2)*x(3)^3, 1*x(1)^3*x(2)*x(3)^3, 1^2*x(1)^2*x(2)*x(3)^3, x(1)^4*x(2)^2*x(3)^2,
//--1*x(1)^3*x(2)^2*x(3)^2, 1^2*x(1)^2*x(2)^2*x(3)^2, 1^3*x(1)*x(3)^4,       1*1^7, 1*1^6*x(1), 1*1^5*x(1)^2, 1*1^4*x(1)^3, 1*1^3*x(1)^4, 1*1^2*x(1)^5,
//--1*1^1*x(1)^6, 1*x(1)^7, x(1)^8;

//---10 monomials+1 monomial--change the last monomial x(1)^a, a=0,1,...,8-----
list lmonx=x(1)^4*x(3)^4, 1*x(1)^3*x(3)^4, 1^2*x(1)^2*x(3)^4, x(1)^4*x(2)*x(3)^3, 1*x(1)^3*x(2)*x(3)^3, 1^2*x(1)^2*x(2)*x(3)^3, x(1)^4*x(2)^2*x(3)^2,
1*x(1)^3*x(2)^2*x(3)^2, 1^2*x(1)^2*x(2)^2*x(3)^2, 1^3*x(1)*x(3)^4,   x(1)^4;
//-------------------------------------------------------------------------------
for (i=1; i<=size(lmonx); i=i+1)
        { wlist=insert(wlist, 1 , size(wlist));}
ring r3=(0,z), (x(1..n+1), t(1..size(lmonx))),wp(wlist[1..n+1+size(lmonx)]);
poly cp=cyclotomic(2*d);   degext=deg(cp) div deg(var(1)); cp=subst(cp, x(1),z);
minpoly =number(cp); //--z is the 2d-th root of unity---
list BasisDR=imap(r2,BasisDR);   list lmonx=imap(r2,lmonx);
list Fn2p1=imap(r2,Fn2p1);
list Periods=imap(r2, Periods);

//---Defining the truncated ideal of the Hodge locus.--------
//---specifying 'size(Periods)'-variables to each zarib would make this faster.
list lII; list II;  intvec z1mosaviz2;
for (i=1; i<=size(Periods); i=i+1)
    {
    lII=insert(lII, HodgeLocusIdeal(mlist, lmonx, Fn2p1, BasisDR, Periods[i], tru,0), size(lII));
    }


list Al; for (i=1; i<=size(Periods); i=i+1){Al=insert(Al, 1,size(Al));} //----list of coefficiens all 1----
poly P; int k; list komak;
int k_1;  int k_2; list lofg; list SR; ideal Ikom; ideal IIi; list SR2;
list IIel;
list redli; list nonredli;
II=list();
         for (k=1; k<=size(Fn2p1); k=k+1)
             {
             komak=list();
             for (j=0; j<=tru; j=j+1)
                 {
                 P=0;
                 for (i=1; i<=size(lII); i=i+1)
                     {
                     P=P+Al[i]*lII[i][k][j+1];
                     }
                 komak=insert(komak, P, size(komak));
                 }
             II=insert(II, komak, size(II));
             }




      //--"Defining a minimal number of generators for the ideal of Hodge locus";
      k=1;
      while (II[k][2]==0){k=k+1;}
      Ikom=II[k][2];
      IIi=II[k][2];
      SR=k;
      for (i=k+1; i<=size(Fn2p1); i=i+1)
          {
          Ikom=std(Ikom);
          if (reduce(II[i][2], Ikom)<>0)
                                     {
                                     Ikom=Ikom, II[i][2]; SR=insert(SR, i, size(SR));
                                     IIi=IIi, II[i][2];
                                     }
          }
      //------------------------------


      "Minimum number of generators:", size(SR);
                       "Minimal generators of the ideal of Hodge locus";
                       for (i=1; i<=size(SR); i=i+1){ Fn2p1[SR[i]];}
                       "Other generators modulo the smaller ideal";

      SR2=list(); for (i=1; i<=size(Fn2p1); i=i+1){SR2=insert(SR2,i, size(SR2));}
      SR2=RemoveList(SR2, SR);
     //-----------------------------
      komak=list(); P=0;   IIel=list();

      for (i=1; i<=size(SR2); i=i+1)
          {
          "The generator:", Fn2p1[SR2[i]];
                  IIel=II[SR2[i]];
                  P=0; lofg=list();

         for (j=2; j<=size(IIel); j=j+1)
            {
                      "Checking ",j-1,"-reducedness.";
                      P=IIel[j]-P; komak=division(P, IIi);   P=0;
                     if (komak[2][1]<>0)
                                  {
                                  "The  Hodge locus for (a_1,a_2)=", Al, "is not", j-1, "-smooth or reduced.";

                                  j=size(IIel)+1; i=size(SR2)+1;  //---getting out of the loop.
                                  }
                     lofg=insert(lofg, komak[1], size(lofg));
                            for (k_1=1; k_1<=size(lofg); k_1=k_1+1)
                                {
                                 for (k_2=1; k_2<=size(SR); k_2=k_2+1)
                                     {
                                      if (j<size(IIel))
                                         {
                                         P=P+lofg[k_1][k_2,1]*II[SR[k_2]][j-k_1+2];
                                         }
                                     }
                                }

                }

            }
\end{verbatim}
}


\def\cprime{$'$} \def\cprime{$'$} \def\cprime{$'$} \def\cprime{$'$}

\end{document}